\documentclass[oneside,english]{amsart}
\usepackage[T1]{fontenc}
\usepackage[latin9]{inputenc}
\usepackage{amsthm}
\usepackage{amstext}
\usepackage{amssymb}
\usepackage{esint}

\makeatletter

\numberwithin{equation}{section}
\numberwithin{figure}{section}
\theoremstyle{plain}
\newtheorem{thm}{\protect\theoremname}[section]
  \theoremstyle{plain}
  
  \theoremstyle{plain}
  
  \theoremstyle{plain}
  \newtheorem{rem}[thm]{\protect\remarkname}
  \theoremstyle{plain}
  \newtheorem{exa}[thm]{\protect\examplename}
  \theoremstyle{plain}
  \newtheorem{lem}[thm]{\protect\lemmaname}
  \theoremstyle{definition}

\makeatother

\usepackage{babel}
  \providecommand{\corollaryname}{Corollary}
  \providecommand{\definitionname}{Definition}
  \providecommand{\lemmaname}{Lemma}
  \providecommand{\propositionname}{Proposition}
  \providecommand{\examplename}{Example}
  \providecommand{\remarkname}{Remark}
\providecommand{\theoremname}{Theorem}

\DeclareMathOperator{\dist}{dist}
\DeclareMathOperator{\diam}{diam}
\DeclareMathOperator{\ess}{ess}

\DeclareMathOperator{\loc}{loc}

\begin{document}

\title{Sobolev homeomorphisms and Brennan's conjecture}

\author{V.~Gol'dshtein and A.~Ukhlov}

\vspace{1cm}

{\it Dedicated to the blessed memory of F.~W.~Gehring}

\vspace{1cm}

\begin{abstract}
Let $\Omega \subset \mathbb{R}^n$ be a domain that supports the $p$-Poincar\'e inequality. Given a homeomorphism $\varphi \in L^1_p(\Omega)$, for $p>n$  we show the domain $\varphi(\Omega)$ has finite geodesic diameter. This result has a direct application to Brennan's conjecture and quasiconformal homeomorphisms. {\bf The Inverse Brennan's conjecture} states that for any simply connected plane domain $\Omega' \subset\mathbb C$ with nonempty boundary and for any conformal homeomorphism $\varphi$ from the unit disc $\mathbb{D}$  onto $\Omega'$  the complex derivative $\varphi'$ is integrable in the degree $s$, $-2<s<2/3$. If $\Omega'$ is bounded than $-2<s\leq 2$.  We prove that integrability in the degree $s> 2$ is not possible for domains $\Omega'$ with infinite geodesic diameter. \end{abstract}

\maketitle


\section{Introduction }

Let $\Omega, \Omega'$ be domains in the Euclidean space $\mathbb{R}^n$, $n\geq 2$, and $\varphi:\Omega \to \Omega'$ be a homeomorphism that belongs to the Sobolev class $L^1_p(\Omega)$. 
Let us remark immediately that for all functions $f$ belong to an equivalence class $[f]\in L^1_p(\Omega)\subset L_{1,\loc}(\Omega)$, $p>n$, we have a unique redefined continuous function $\tilde{f}\in L^1_p(\Omega)$ (see, for example \cite{HH, Z}). This redefined function we take as an element of Sobolev space $L^1_p(\Omega)$, $p>n$.

The main result of this work is the following statement:
\begin{thm} \label{main}
Let $\Omega$ and $\Omega'$ be domains in $\mathbb{R}^n$, and  $\Omega$ supports Poincar\'e $p$-inequality for some $p>n$. Suppose that there exists a Sobolev homeomorphism $\varphi : \Omega\to \Omega'$ of the class $L^1_p(\Omega)$ . Then the domain $\Omega'$ has a finite geodesic diameter.
\end{thm}

Recall {\bf the Inverse Brennan's conjecture} for conformal homeomorphisms. Let $\Omega\subset\mathbb C$ be a simply connected plane domain with nonempty boundary, $\mathbb{D}$ be the unit disc and $\varphi: \mathbb D\to\Omega$ be a conformal homeomorphism. The Inverse Brennan's conjecture \cite{Br} states that  
$$
\iint\limits_{\mathbb D}|\varphi'(z)|^s~dxdy<+\infty,\quad \text{for all}\quad -2<s<\frac{2}{3}.
$$
The example $\Omega = \mathbb C \setminus (-\infty,-1/4]$
shows that this range of $s$ cannot be extended.

\vspace{0.3cm}

The first observation about an extended range of integrability for conformal homeomorphisms belongs to F.~W.~Ghering (was not published) \cite{Br}.

\vspace{0.3cm}

Recall one of classical definitions of conformal homeomorphisms. A homeomorphism $\varphi: \Omega\to\Omega'$, $\Omega$ and $\Omega'$ domains in $\mathbb C$, is called conformal if $\varphi$ is a diffeomorphism and 
$$
|\varphi'(z)|^2=J(z,\varphi) \,\,\,\text{for all}\,\,\,z\in\Omega.
$$

In the general form the Inverse Brennan's conjecture is proved for the range $-1.78<p<2/3$, see, for example, \cite{Ber}.
The Brennan's conjecture has applications to weighted Sobolev embedding theorems and solvability of elliptic equations \cite{GU1, GU2, GU3}. 
In these works was proved that sharpness of embeddings depends on the range of $s$ in the Inverse Brennan's conjecture.
On this way the problem to describe classes of domains, for which this range of integrability can be extended arises.

We start from the class of domains with bounded measure (area).
If the Lebesgue measure (area) $\mu_2(\Omega)$ of the domain $\Omega$ is finite, then, by simple calculations we obtain
$$
\iint\limits_{\mathbb D}|\varphi'(z)|^2~dxdy=\iint\limits_{\mathbb D}J(z,\varphi)~dxdy=\iint\limits_{\Omega}~dudv=m_2(\Omega)<\infty.
$$
Hence, for domains with finite measure the conjecture holds for $s\in [s_0,2]$. Here $s_0$ is the best lower bound for which Inverse Brennan's conjecture is proved. 

We show that this upper bound $s=2$ can not be extended without additional assumptions on $\Omega$:

\begin{thm}
\label{thm:Base} Let for a conformal homeomorphism $\varphi:\mathbb D\to\Omega$ the complex derivative $\varphi'$ is integrable in the degree $s>2$. Then the domain $\Omega$ has finite geodesic diameter.
\end{thm} 

This theorem is a direct consequence of the theorem \ref{main} and permit us to formulate the Inverse Brennan's conjecture for domains with finite measure (area). 

{\bf Inverse Brennan's conjecture for domains with finite measure}. This conjecture concerns integrability of derivatives of plane conformal homeomorphisms $\varphi:\mathbb D \to  \Omega$ of the unit disc $\mathbb D\subset\mathbb R^2$ to a simply connected plane domain with finite measure (area). 

The conjecture
\begin{equation}
\iint\limits_{\mathbb D}|\varphi'(z)|^s~dxdy<+\infty,\quad \text{for all}\quad -2<s \leq 2.
\end{equation}

The example $\Omega=\mathbb D \setminus (-1,-1/4]$
shows that the lower limit $-2$ of $s$ cannot be extended. From the previous theorem follows that $s\leq 2$ for any 
simply connected plane domain with finite measure (area) and infinite geodesic diameter. 

\vskip 0.2cm
{\bf Open problems:} {\bf 1.} Does estimate $-1.78<s$ can be improved for 
simply connected plane domain with finite measure (area)?

{\bf 2.} Does estimates $s \leq 2$  can be improved for for 
simply connected plane domain with finite measure (area) and finite geodesic diameter or
under some other reasonable geometric condition?
\vskip 0.2cm

The proposed method is based on the composition operators theory (see, for example, \cite{GR, U1, VU1}) which allows to "transfer" the Poincar\'e (Sobolev) inequalities from one domain to another \cite{GG,GU01, GU3}. 

The following diagram illustrate this idea:

\[\begin{array}{rcl}
L^1_{\infty}(\Omega') & \stackrel{\varphi^*}{\longrightarrow} & L^1_p(\Omega) \\[2mm]
\multicolumn{1}{c}{\downarrow} & & \multicolumn{1}{c}{\downarrow} \\[1mm]
C(\Omega') & \stackrel{(\varphi^{-1})^*}{\longleftarrow} & C(\Omega)
\end{array}\]

Here the operator $\varphi^{\ast}$ defined by the composition rule $\varphi^{\ast}(f)=f\circ\varphi$ is a bounded composition operator of Sobolev spaces induced by a homeomorphism $\varphi$ of domains $\Omega$ and $\Omega'$.

An another domain of applications of the theorem \ref{main} is the integrability problem for derivatives of $n$-dimensional quasiconformal homeomorphisms. This problem was extensively studied (see, for example, in \cite{AK, G1, MV}).

We show that the similar to conformal case result about global integrability of derivatives of $n$-dimensional quasiconformal homeomorphisms is also a direct consequence of the theorem \ref{main}:

\begin{thm}
\label{thm:BaseQC} Let for a quasiconformal mapping $\varphi:\mathbb B\to\Omega$, $\Omega\subset\mathbb R^n$, $\mathbb B\subset\mathbb R^n$ is the unit ball, the derivative $D\varphi$ is integrable in the degree $s>n$. Then the domain $\Omega$ has finite geodesic diameter.
\end{thm} 

\section{The Poincar\'e type inequality}

Let $\Omega\subset\mathbb{R}^n$ be a $n$-dimensional Euclidean domain.
As usually Lebesgue space $L_p(\Omega)$, $1\leq p\leq\infty$, is the space of locally summable functions with the finite norm:
$$
\| f\mid L_p(\Omega)\|=\biggr(\int\limits_\Omega |f(x)|^p~dx\biggl)^{\frac{1}{p}},\quad 1\leq p<\infty,
$$
and
$$
\|f\mid L_{\infty}(\Omega)\|=
\ess\sup\limits_{x\in \Omega}|f(x)|,\quad p=\infty.
$$

The space of bounded continuous on $\Omega$ functions $C(\Omega)$ considered with the norm:
$$
\|f\mid C(\Omega)\|=
\sup\limits_{x\in \Omega}|f(x)|.
$$

The space $C_0(\Omega)\subset C(\Omega)$ is a linear subspace of continuous functions with compact support in $\Omega$.

Here we reproduce a standard definition of seminormed Sobolev spaces $L^1_p(\Omega)$.
 
The seminormed Sobolev space
$L^1_p(\Omega)$, $1\leq p\leq\infty$, consists of locally summable, weakly differentiable functions $f:\Omega\to\mathbb R$ with the finite seminorm:
$$
\|f\mid L^1_p(\Omega)\|=\| |\nabla f|\mid L_p(\Omega)\|.
$$

Note that every element $[f]\in L^1_p(\Omega)\subset L_{1,\loc}(\Omega)$ is an equivalence class 
of locally integrable functions that coincide a.~e.~in $\Omega$. In the case $p>n$ the Lebesgue redefinition $\tilde{f}$ of a function $f\in [f]$:
$$
\tilde{f}(x)=\lim\limits_{r\to 0}\frac{1}{B(x,r)}\int\limits_{B(x,r)}f(y)~dy,\,\,\, f\in [f],
$$
is defined at every point $x\in \Omega$ (see, for example \cite{HH, Z}) and by the Sobolev embedding theorem is a continuous function. Hence for all functions $f\in [f]$ we have a unique redefined continuous function $\tilde{f}\in L^1_p(\Omega)$ and we take this function as an element of Sobolev space $L^1_p(\Omega)$ for $p>n$.

A mapping $\varphi:\Omega\to\mathbb{R}^n$ belongs to $L^1_p(\Omega)$, $1\leq p\leq\infty$, if
its coordinate functions belong to $L^1_p(\Omega)$.
In this case the formal Jacobi matrix $D\varphi (x)$
and its determinant (Jacobian) $J(x,\varphi)$ are well defined at
almost all points $x\in \Omega$. The norm $|D\varphi (x)|$ of the matrix
$D\varphi (x)$ is the norm of the corresponding linear operator. We will use the same notation for this matrix and the corresponding linear operator.

We say that a domain $\Omega$ supports  $p$-Poincar\'e in\-equa\-li\-ty for $p>n$ if
any function $f\in L^1_p(\Omega)$ belongs to $C(\Omega)$ and the following inequality 
$$
\inf_{c\in \mathbb R} \|f-c\mid C(\Omega)\|\leq K \|f\mid L^1_{p}(\Omega)\|
$$
holds. Here the constant $K$ does not depends on $f$. 

We study existence conditions for Sobolev homeomorphisms of a domain $\Omega$ that supports $p$-Poincar\'e inequality for $p>n$ onto a domain $\Omega'$ in terms of existence of $p$-Poincar\'e inequalities in the domain $\Omega'$.

For example, every bounded domain with a smooth boundary supports  the $p$-Poincar\'e for any  $p>n$.

\begin{thm}
\label{thm:Poincare} 
Let $\Omega$ and $\Omega'$ be domains in $\mathbb{R}^n$ and the domain $\Omega$ supports $p$-Poincar\'e inequality for $p>n$. Suppose that there exists a homeomorphism $\varphi : \Omega\to \Omega'$ that belongs to $L^1_p(\Omega)$. Then every function $f\in L^1_{\infty}(\Omega')$  belongs to  $C(\Omega')$ and the following inequality
\begin{equation}\label{PoIn}
\inf_{c\in \mathbb R} \|f-c\mid C(\Omega')\|\leq K \|f\mid L^1_{\infty}(\Omega')\|
\end{equation}
holds. Here constant $K$ does not depends on $f$.
\end{thm}

\begin{proof} 
Let $f \in L^1_{\infty}(\Omega')$. Then by \cite{GU4,GU5} the composition operator 
$$
\varphi^{\ast}: L^1_{\infty}(\Omega')\to L^1_p(\Omega),\quad \varphi^{\ast}(f)=f\circ\varphi,
$$
is bounded and the inequality
$$
\|\varphi^{\ast}(f)\mid L^1_{p}(\Omega)\|\leq A\|f\mid L^1_{\infty}(\Omega')\|
$$
holds for any function  $f \in L^1_{\infty}(\Omega')$.

Hence the function $\varphi^{\ast}(f)$ belongs to the Sobolev space $L^1_p(\Omega)$. Since the domain $\Omega$ supports  the $p$-Poincar\'e inequality  we have the following inequality 
$$
\inf_{c\in \mathbb R} \|\varphi^{\ast}(f)-c\mid C(\Omega)\|\leq M \|\varphi^{\ast}(f)\mid L^1_{p}(\Omega)\|
$$
for any $f \in L^1_{\infty}(\Omega')$. Here the constant $M$ does not depends on $f$.

Because $\varphi^{-1}$ is a homeomorphism   it induces an isometry of spaces $C(\Omega)$ and $C(\Omega')$.  Then $f=(\varphi^{-1})^{\ast}(\varphi^{\ast}(f)) \in C(\Omega')$  and
$$
\|f-c\mid C(\Omega')\|=\|\varphi^{\ast}(f)-c\mid C(\Omega)\|
$$
for every $c\in\mathbb R$.

Because
$$
\varphi^{\ast}: L^1_{\infty}(\Omega')\to L^1_p(\Omega)
$$ is a bounded composition operator we have finally
\begin{multline}
\inf_{c\in \mathbb R} \|f-c\mid C(\Omega')\|=\inf_{c\in \mathbb R}\|\varphi^{\ast}(f)-c\mid C(\Omega)\|\\
\leq M \|\varphi^{\ast}(f)\mid L^1_{p}(\Omega)\|\leq AM\|f\mid L^1_{\infty}(\Omega')\|=K\|f\mid L^1_{\infty}(\Omega')\|.
\nonumber
\end{multline}
\end{proof}
\vskip 0.3cm

Let us illustrate this theorem with help of a simple plane example. We identify the Euclidean space $\mathbb R^2$ and the complex plane $\mathbb C$ putting $z=x+iy\in\mathbb C$ where $(x,y)\in\mathbb R^2$.

\begin{exa}
Let $\Omega_s=\{z=x+iy\in \mathbb C : 1< x<\infty,0< y< 1\}$ be the plane domain and $\mathbb D\subset\mathbb C$ be the unit disc. Then does not exists a Sobolev homeomorphism $\varphi: \mathbb D\to\Omega_s$ of the class $L^1_p(\mathbb D)$ for any  $p>2$.
\end{exa}

\begin{proof} 
Consider the function $f(x,y)=\frac{1}{\sqrt{2}}(x+y)$, defined on the domain $\Omega$. Then $|\nabla f(x,y)|=1$ for all $z=x+iy\in\Omega$ and 
$\|\nabla f\mid L_{\infty}(\Omega)\|=1$. From the other side, $f(x,y)\rightarrow \infty$, as $x\rightarrow\infty$. Hence for the function $f(x,y)$ the inequality (\ref{PoIn})  doesn't holds and by Theorem \ref{thm:Poincare}  for any $p>2$ a homeomorphism $\varphi: \mathbb D\to\Omega$  of the class $L^1_p(\mathbb D)$ can not exists.
\end{proof} 

In the case $p=2$, $n=2$ we have following simple necessary condition for existence of a Sobolev homeomorphisms $\varphi:\mathbb D\to\Omega$ of the class 
$L^1_2(\mathbb D)$.

\vskip 0.3cm

\begin{lem} 
Let a homeomorphism $\varphi : \Omega\to \Omega'$ belongs to $L^1_2(\Omega)$, $\Omega,\Omega'\subset\mathbb C$. Then the domain $\Omega'$ has finite measure.
\end{lem}

\begin{proof}
By change of variable formula for Sobolev mappings (see, for example, \cite{GR}) and H\"older inequality we have
\begin{multline}
m_2(\Omega')=\\=\iint\limits_{\Omega} |J(z,\varphi)|~dxdy=
\iint\limits_{\Omega} |u_xv_y-u_yv_x|~dxdy
\leq\iint\limits_{\Omega} |u_xv_y|~dxdy+\iint\limits_{\Omega} |u_yv_x|~dxdy\\
\leq\left(\iint\limits_{\Omega} |u_x|^2~dxdy\right)^{\frac{1}{2}}\left(\iint\limits_{\Omega} |v_y|^2~dxdy\right)^{\frac{1}{2}}+\left(\iint\limits_{\Omega} |u_y|^2~dxdy\right)^{\frac{1}{2}}\left(\iint\limits_{\Omega} |v_x|^2~dxdy\right)^{\frac{1}{2}}\\
\leq\left(\iint\limits_{\Omega} |\varphi'|^2~dxdy\right)^{\frac{1}{2}}\left(\iint\limits_{\Omega} |\varphi'|^2~dxdy\right)^{\frac{1}{2}}+\left(\iint\limits_{\Omega} |\varphi'|^2~dxdy\right)^{\frac{1}{2}}\left(\iint\limits_{\Omega} |\varphi'|^2~dxdy\right)^{\frac{1}{2}}\\
= 2\left(\iint\limits_{\Omega} |\varphi'|^2~dxdy\right) =2\|\varphi\mid L^1_2(\Omega)\|^2.
\nonumber
\end{multline}
\end{proof}

In the next example we construct a  domain of finite measure $\Omega_{\alpha}\subset\mathbb C$ which doesn't allows a Sobolev homeomorphism 
$\varphi: \mathbb D\to\Omega_{\alpha}$ of  class $L^1_p(\mathbb D)$, $p>2$, of the unit disc $\mathbb D\subset \mathbb C$ onto the domain $\Omega_{\alpha}\subset\mathbb C$. 

\begin{exa}
Does not exists a Sobolev homeomorphism $\varphi \in L^1_p(\mathbb D)$ of the unit disc $\mathbb D \subset\mathbb C$ onto the domain
$$
{\Omega}_{\alpha}=\{z=x+iy\in \mathbb C : 1< x<\infty,0< y< 1/x^{\alpha},\,\,\alpha>1\}.
$$
for any $p>2$.
\end{exa}

\begin{proof}
Let us remark that $m_2(\Omega_{\alpha})<\infty$, because $\alpha>1$ and 
$$
m_2(\Omega_{\alpha})=\int\limits_{\Omega_{\alpha}}dxdy=\int\limits_1^{\infty}\frac{dx}{x^{\alpha}}=\frac{1}{\alpha-1}.
$$

Consider the function $f(x,y)=\frac{1}{\sqrt{2}}(x+y)$, defined on the domain ${\Omega}_{\alpha}$. Then $|\nabla f(x,y)|=1$ for all $z=x+iy\in\Omega$ and 
$\|\nabla f\mid L_{\infty}({\Omega}_{\alpha})\|=1$. From the other side, $f(x,y)\rightarrow \infty$, as $x\rightarrow\infty$. Hence for the function $f(x,y)$ the inequality (\ref{PoIn})  doesn't holds and by Theorem \ref{thm:Poincare}  for any $p>2$ a homeomorphism $\varphi: \mathbb D\to\Omega$  of the class $L^1_p(\mathbb D)$ can not exists.
\end{proof}

\section{Domains with finite geodesic diameter}

We define the geodesic diameter $\diam_G(\Omega)$ of a domain $\Omega\subset\mathbb{R}^n$ as
$$
\diam_G(\Omega)=\sup\limits_{x,y\in\Omega}\dist_{\Omega}(x,y).
$$
Here $\dist_{\Omega}(x,y)$ is the intrinsic geodesic distance:
$$
\dist_{\Omega}(x,y)=\inf_{\gamma\in\Omega}\int\limits_0^1 |\gamma'(t)|~dt
$$
where infimum is taken over all rectifiable curves $\gamma\in\Omega$ such that $\gamma(0)=x$ and $\gamma(1)=y$.

The main result of this work is the following:

\begin{thm}\label{thm:FinDiam} 
Let $\Omega$ and $\Omega'$ be domains in $\mathbb{R}^n$ and $\Omega$ supports $p$-Poincar\'e inequality for some $p>n$.  Suppose that there exists a Sobolev homeomorphism $\varphi : \Omega\to \Omega'$ of the class $L^1_p(\Omega)$. Then the domain $\Omega'$ has a finite geodesic diameter.
\end{thm}

\begin{proof}
Let $x_0$ be an arbitrary fixed point in the domain $\Omega'$. We consider the function $f(x)=\dist_{\Omega'}(x_0,x)$, $x\in\Omega'$. By definition the function $f$ is a Lipschitz function and by the Rademacher Theorem $f$ is differentiable almost everywhere.

Moreover $|\nabla f|=1$ almost everywhere on $\Omega'$ and because $f$ is a Lipschitz function it is a weakly differentiable and it belongs to the Sobolev space $L^1_{\infty}(\Omega')$. Hence, by Theorem \ref{thm:Poincare} the function $f$ belongs to the space $C(\Omega')$ and 
$$
\inf_{c\in \mathbb R} \|f-c\mid C(\Omega')\|\leq K \|f\mid L^1_{\infty}(\Omega')\|.
$$
Here a constant $K$ does not depends on $f$.
It means that there exists a number $c_0$ such that 
$$
\|f-c_0\mid C(\Omega')\|\leq 2K \|f\mid L^1_{\infty}(\Omega')\|
$$

Therefore we have the following estimate:
$$
\sup\limits_{x\in \Omega'}\dist_{\Omega'}(x_0,x)=\|f\mid C(\Omega')\|\leq 2K \|f\mid L^1_{\infty}(\Omega')\|+|c_0|.
$$
Because 
$$ 
\diam_G(\Omega)\leq 2 \sup\limits_{x\in \Omega'}\dist_{\Omega'}(x_0,x)
$$
we have finally
$$
\diam_G(\Omega)\leq 4 \{ K \|f\mid L^1_{\infty}(\Omega')\|+|c_0| \}.
$$
\end{proof}

\begin{exa} We reproduce here the classical example of a bounded plane domain  with infinite geodesic diameter.

Consider the ring domain $\Omega_R=\{z=x+iy\in \mathbb C : 1<x^2+y^2<4 \}$. Fix a number $0<r<1/4$ and for natural number $n$, $n=1,2,...$, consider circles:
$$
S^r_{2n}=\biggl\{z=x+iy\in \mathbb C : x^2+y^2=\bigl(1+\frac{1}{2n}\bigr)^2, \quad x<1-r \biggr\}
$$
and
$$
S^r_{2n+1}=\biggl\{z=x+iy\in \mathbb C : x^2+y^2=\bigl(1+\frac{1}{2n+1}\bigr)^2, \quad x>r-1 \biggr\}.
$$

Let 
$$
\Omega=\Omega_{R}\setminus (\cup_{n=1}^{\infty} S^r_{2n})\cup (\cup_{n=1}^{\infty} S^r_{2n+1}).
$$
Then  for every $\varepsilon>0$ does not exist a conformal homeomorphism $\varphi$ of unit disc $\mathbb{D}$ onto domain $\Omega$ of the class $L^1_{2+\varepsilon}(B)$. 
\end{exa}

\begin{proof} By the construction this domain $\Omega_R$ has an infinite geodesic diameter and finite measure. Then by the previous theorem does not exists a Sobolev homeomorphism $ \varphi: \mathbb{D} \to \Omega $ of the class $L^1_{2+\varepsilon}(\mathbb{D})$. 
By the Riemann Mapping Theorem there exist a conformal homeomorphism $\varphi: \mathbb{D} \to \Omega $. Because $\Omega$ is bounded domain this homeomorphism belongs to the class $L^1_{2}(\mathbb{D})$ and do not belongs to the class $L^1_{2+\varepsilon}(\mathbb{D})$ for any $\varepsilon>0$. 
\end{proof}

As a direct consequence of Theorem \ref{thm:FinDiam} we immediately obtain  the following assertion about global integrability of  conformal homeomorphisms $\varphi: \mathbb{D} \to \Omega $ in degree $p>2$ .

\begin{thm}
For any $\varepsilon>0$ does not exist a conformal homeomorphism $\varphi\in L^1_{2+\varepsilon}(\mathbb D)$ of the unit disc $\mathbb D\subset\mathbb C$ onto a plane domain $\Omega\subset\mathbb C$ with infinite geodesic diameter.
\end{thm}

\begin{rem} 
Remind that by Riemann Mapping Theorem for any simply connected plane domain $\Omega$ there exists a conformal homeomorphism $\varphi : \mathbb D\to \Omega$ but this conformal homeomorphism can not belong to the class $L^1_{2+\varepsilon}(B)$ for any $\varepsilon>0$ if $\Omega$ has infinite geodesic diameter. 
\end{rem}

The similar results for quasiconformal mappings follows from our main theorem: 
\begin{thm}
For any $\varepsilon>0$ does not exist a quasiconformal homeomorphism  $\varphi\in L^1_{n+\varepsilon}(\mathbb B)$ of the unit ball $\mathbb B\subset\mathbb R^n$ onto a domain $\Omega\subset\mathbb R^n$ with infinite geodesic diameter.
\end{thm}

\noindent
Vladimir Gol'dshtein  \,  \hskip 3.2cm Alexander Ukhlov

\noindent
Department of Mathematics   \hskip 2.25cm Department of Mathematics

\noindent
Ben-Gurion University of the Negev  \hskip 1.05cm Ben-Gurion University of the Negev

\noindent
P.O.Box 653, Beer Sheva, 84105, Israel  \hskip 0.7cm P.O.Box 653, Beer Sheva, 84105, Israel

\noindent
E-mail: vladimir@bgu.ac.il  \hskip 2.5cm E-mail: ukhlov@math.bgu.ac.il

\end{document}